\documentclass[12pt]{amsart}
\usepackage{amsfonts}
\usepackage{amsfonts,latexsym,rawfonts,amsmath,amssymb,amsthm}
\usepackage[plainpages=false]{hyperref}
\usepackage{graphicx}

\numberwithin{equation}{section}

\RequirePackage{color}
\theoremstyle{plain}

\newtheorem{theorem}{Theorem}[section]

\newtheorem{definition}{Definition}[section]

\newtheorem{lemma}[theorem]{Lemma}
\newtheorem{proposition}[theorem]{Proposition}

\newcommand{\beq}{\begin{equation}}
\newcommand{\eeq}{\end{equation}}
\newcommand{\beqs}{\begin{eqnarray*}}
\newcommand{\eeqs}{\end{eqnarray*}}
\newcommand{\beqn}{\begin{eqnarray}}
\newcommand{\eeqn}{\end{eqnarray}}
\newcommand{\beqa}{\begin{array}}
\newcommand{\eeqa}{\end{array}}

\def\phi{\varphi}

\begin{document}
\title{A Class of prescribed Weingarten curvature equations in Euclidean spcace}

\author{Li Chen}
\address{Faculty of Mathematics and Statistics, Hubei Key Laboratory of Applied Mathematics, Hubei University,  Wuhan 430062, P.R. China}
\email{chenli@hubu.edu.cn}

\author{Agen Shang}
\address{Faculty of Mathematics and Statistics, Hubei Key Laboratory of Applied Mathematics, Hubei University,  Wuhan 430062, P.R. China}
\email{shangagen\_2019@163.com}

\author{Qiang Tu}
\address{Faculty of Mathematics and Statistics, Hubei Key Laboratory of Applied Mathematics, Hubei University,  Wuhan 430062, P.R. China}
\email{qiangtu@hubu.edu.cn}



\thanks{This research was supported by funds from Hubei Provincial Department of Education
Key Projects D20181003.}

\begin{abstract}
In this paper, we consider a class of prescribed Weingarten
curvature equations. Under some sufficient condition, we obtain an
existence result by the standard degree theory based on the a prior
estimates for the solutions to the prescribed Weingarten curvature
equations.
\end{abstract}

\maketitle

\baselineskip18pt

\parskip3pt

 \section{Introduction}

Let $(M, g)$  be  a smooth, compact Riemannian manifold of
dimension $n\ge3$ and $V$ be a $(0,2)$ tensor on $(M, g)$. The $\sigma_k$- curvature of $V$ is defined by
\begin{equation*}
\sigma_k(V),
\end{equation*}
where $\sigma_k(V)$ means $\sigma_k$ is applied to the eigenvalues
of $g^{-1} V$ and the $k$-th elementary symmetric polynomial $\sigma_k$ is defined by:
\[\sigma_k(\lambda)=\sum_{1\le i_1<\cdots<i_k\le n}\lambda_{i_1}\cdots\lambda_{i_k}.\]

Recently, the following fully nonlinear equations of linear combination of $\sigma_k$-curvature of $V$
\begin{eqnarray}\label{V}
\sigma_k(V)=\sum_{l=0}^{k-1}\alpha_l(x) \sigma_{l}(V), \quad 2\leq k\leq n.
\end{eqnarray}
are widely studied, where $\alpha_l(x)$ are given smooth functions
on $M$. For $V=D^2 u$, Krylov \cite{Kry95} considered  Dirichlet
problem of the following degenerate equation in a $(k-1)$-convex
domain $D\subset\mathbb{R}^n$,
\[\sigma_k(D^2u) =
\sum_{l=0}^{k-1} \alpha_l(x)\sigma_l(D^2u)\] with all coefficient
$\alpha_{l}(x)\geq 0$ for $0\leq l\leq k-1$. Recently, in
\cite{GZ19} the authors studied
 \begin{eqnarray*}\label{p}
\sigma_k(D^2u+uI)=\sum_{l=0}^{k-1}\alpha_l(x)\sigma_l(D^2u+uI) \quad \mbox{on} \quad \mathbb{S}^n,
\end{eqnarray*}
which arises in the problem of prescribed convex combination of area
measures \cite{Sch13}. Motivated by \cite{Kry95} and \cite{GZ19},
the authors in \cite{Chen19} studied the equations of linear
combination of the prescribed $\sigma_k$-scalar curvature
\begin{equation*}
\sigma_k(\frac{-Ric_g}{n-2})+\alpha(x)\sigma_{k-1}(\frac{-Ric_g}{n-2})=f(x),
\quad 3\leq k\leq n,
\end{equation*}
and
\begin{eqnarray*}
\sigma_k( A_{{g}}) +\alpha(x) \sigma_{k-1}(A_{{g}})=f(x), \quad 3\leq k\leq n,
\end{eqnarray*}
where ${Ric}_{g}$ and
$${A}_{g}
={\frac{1}{{n-2}}}
\left({Ric}_{g} - {\frac{{{R}_{g}}}{{2(n-1)}}}g
\right)$$ are the Ricci curvature and the Schouten tensor
of $g$ respectively.

Because of its structure as a combination of elementary symmetric functions, equation
\eqref{V} is also interesting from the PDE point of view. Such type of equations arise naturally
from many important geometric problems, such as the so-called Fu-Yau equation
arising from the study of the Hull-Strominger system in theoretical physics, which is an
equation that can be written as the linear combination of the first and second elementary
symmetric functions,
\begin{eqnarray*}
\sigma_1(i\partial \overline{\partial}(e^u
+\alpha^{\prime}e^{-u})) +\alpha^{\prime}\sigma_{2}(i\partial \overline{\partial}u)=\phi
\end{eqnarray*}
on $n$-dimensional compact K$\ddot{a}$hler manifolds, see Fu-Yau \cite{Fu07, Fu08}
and Phong-Picard-Zhang \cite{Ph17, Ph19, Ph20}. Furthermore, the special
Lagrangian equations introduced by Harvey and
Lawson \cite{HL} can also be written as the alternative combinations of elementary symmetric
functions,
\begin{eqnarray*}
\sin\theta\sum_{k=0}^{[\frac{n}{2}]}(-1)^{k}\sigma_{2k}(D^2u)+\cos\theta
\sum_{k=0}^{[\frac{(n-1)}{2}]}(-1)^{k}\sigma_{2k+1}(D^2u)=0.
\end{eqnarray*}
Moreover, equations of the form
\begin{eqnarray*}
\sigma_{1}(D^2u)+b\sigma_{n}(D^2u)=C
\end{eqnarray*}
for some constants $b\geq0$ and $C>0$ also arise from the study of $J$-equation on toric
varieties by Collins-Sz$\acute{e}$kelyhidi \cite{Co17}.

In this paper, we study the problem of prescribed Weingarten
curvature which consists of linear combination of $k$-mean curvature
\begin{eqnarray}\label{Eq}
\sigma_k(\kappa(X))=\sum_{l=0}^{k-1}\alpha_l(X) \sigma_{l}(\kappa(X)), \quad 2\leq k\leq n,
\end{eqnarray}
on a closed Riemannian manifold $M$, where $\kappa(X)=(\kappa_1(X),
..., \kappa_n(X))$ are the principle curvatures of hypersurface $M$
which is an embedded manifold in $\mathbb{R}^{n+1}$ at $X$.
$\sigma_k$, $k=1, 2, ..., n,$ are the Weingarten curvatures of
hypersurface $M$ (or $k$-mean curvature). In the cases $k=1, 2$ and
$n$, they are the mean curvature, scalar curvature, and Gauss
curvature, respectively.

To ensure the ellipticity of \eqref{Eq}, we have to restrict the class of functions
and domains.
\begin{definition}
A smooth hypersurface $M \subset \mathbb{R}^{n+1}$ is called
$k$-convex if the principle curvature vector of $M$
$$\kappa(X)=(\kappa_1(X), ..., \kappa_n(X))$$ belongs
to $\Gamma_k$ for all $X \in M$, where $\Gamma_k$ is the Garding's
cone
\begin{eqnarray*}\label{cone}
\Gamma_{k}=\{\lambda \in \mathbb{R} ^n: \sigma_{j}(\lambda)>0, \forall 1\leq j \leq k\}.
\end{eqnarray*}
\end{definition}

We mainly get the following theorem.

\begin{theorem}\label{Main}
Let $n\ge3, k\ge 2$ and $\alpha_l(X)\in C^\infty(\mathbb{R}^{n+1})$ be positive functions for
all $0\leq l\leq k-1$, assume that
\begin{eqnarray}\label{ASS1}
\frac{\sigma_k(e)}{|X|^{k}}
\geq \sum_{l=0}^{k-1}\alpha_{l}(X)\frac{\sigma_{l} (e)}{|X|^{l}}
\quad \mbox{for} \quad |X|\geq r_2,
\end{eqnarray}
\begin{eqnarray}\label{ASS2}
\frac{\sigma_k(e)}{|X|^{k}}
\leq \sum_{l=0}^{k-1}\alpha_{l}(X)\frac{\sigma_{l} (e)}{|X|^{l}} \quad \mbox{for} \quad |X|\leq r_1,
\end{eqnarray}
and
\begin{eqnarray}\label{ASS3}
\frac{\partial }{\partial \rho}\bigg[\rho^{k-l}\alpha_l(X)\bigg]\leq
0 \quad \mbox{for} \quad r_1\leq |X|\leq r_2,
\end{eqnarray}
where $\rho=|X|$ and $e=(1,1,...,1)$. Then there exists a
$k$-convex, star-shaped hypersurface $M$ in $\{r_1\leq |X|\leq
r_2\}$ satisfies equation (\ref{Eq}).
\end{theorem}

If $\alpha_l\equiv 0$ for $1 \leq l\leq k-1$, the equation
\eqref{Eq} is just the prescribed Weingarten curvature equation
\begin{eqnarray}\label{wce}
\sigma_{k}(\kappa(X))=f(X),
\end{eqnarray}
which has been widely studied in the past two decades. Such results
were obtained for case of prescribed mean curvature by
Bakelman-Kantor \cite{Ba1, Ba2} and by Treibergs-Wei \cite{Tr}. For
the case of prescribed Gaussian curvature by Oliker \cite{Ol}. For
general Weingarten curvatures by Aleksandrov \cite{Al}, Firey
\cite{Fi}, Caffarelli-Nirenberg-Spruck \cite{Ca} for a general class
of fully nonlinear operators $F$, including $F=\sigma_k$ and
$F=\frac{\sigma_k}{\sigma_l}$. Some results have been obtained by
Li-Oliker \cite{Li-Ol} on unit sphere, Barbosa-de Lira-Oliker
\cite{Ba-Li} on space forms, Jin-Li \cite{Jin} on hyperbolic
space, Andrade-Barbosa-de Lira \cite{An} on warped product
manifolds, that is $N=\mathbb{R}\times \Sigma,
g=dr^2+\phi^{2}(r)g_{S^n}$, Li-Sheng \cite{Li-Sh} for Riemannain
manifold equipped with a global normal Gaussian coordinate system,
that is $N=\mathbb{R}\times \Sigma, g=dr^2+\sigma_{ij}(x, r)dx_i
dx_j$.

For the equation \eqref{wce} in the case that $f=f(X, \nu(X))$, that is
$f$ is dependent of $X$ and the normal vector field $\nu$ along the
hypersurface $M$, in many cases, the curvature estimates are the key
part for this prescribed curvature problem. When $k=1$, curvature
estimate comes from the theory of quasilinear PDE. If $k=n$,
curvature estimate is due
to Caffarelli-Nirenberg-Spruck \cite{Ca}. Ivochkina \cite{Iv1, Iv2}
considered the Dirichlet problem of the above equation on domains in
$\mathbb{R}^n$, and obtained $C^2$ estimates there under some extra
conditions on the dependence of $f$ on $\nu$. $C^2$ estimate was
also proved for equation of prescribing curvature measures problem
in \cite{Guan09, Guan12}, where $f(X, \nu)=\langle X, \nu\rangle
\widetilde{f}(X)$. For $f=f(\nu)$ dependent only on $\nu$, the
$C^2$-estimate was proved in B. Guan and P. Guan \cite{Guan02}.
Recently, Guan, Ren and Wang \cite{Guan-Ren15} obtained global $C^2$
estimates for a closed convex hypersurface $M\subset \mathbb{R}^{n+1}$
and then solved the long standing problem \eqref{wce}. In the same
paper \cite{Guan-Ren15}, they also proved the estimate for
starshaped 2-convex hypersurfaces by introducing some new test
curvature functions. In \cite{Li-Ren}, Li, Ren and Wang relax the
convex to $k+1$-convex for any $k$ Hessian equations. In \cite{Ren},
Ren and Wang totally solved the case $k=n-1$, that is the global
curvature estimates of $n-1$ convex solutions of \eqref{wce} and
Hessian equations for $k=n-1$. In \cite{Sp}, Spruck-Xiao extended
2-convex case in \cite{Guan-Ren15} to space forms and give a simple
proof for the Euclidean case. Recently, Chen ,Li and Wang
\cite{Chen} extended \cite{Guan-Ren15} and \cite{Ren} to warped
product manifolds.

The organization of the paper is as follows.
In Sect. 2
we start with some preliminaries.
$C^0$, $C^1$ and $C^2$ estimates are given in Sect. 3.
In Sect. 4 we prove theorem\ref{Main}.

\section{Preliminaries}

\subsection{Setting and General facts}
For later convenience, we first state our conventions on Riemann
Curvature tensor and derivative notation. Let $M$ be a smooth
manifold and $g$ be a Riemannian metric on $M$ with Levi-Civita
connection $\nabla$. For a $(s, r)$ tensor field $\alpha$ on $M$,
its covariant derivative $\nabla \alpha$ is a $(s, r+1)$ tensor
field given by
\begin{eqnarray*}
&&\nabla \alpha(Y^1, .., Y^s, X_1, ..., X_r, X)
\\&=&\nabla_{X} \alpha(Y^1, .., Y^s, X_1, ..., X_r)\\&=&X(\alpha(Y^1, .., Y^s, X_1, ..., X_r))-
\alpha(\nabla_X Y^1, .., Y^s, X_1, ..., X_r)\\&&-...-\alpha(Y^1, ..,
Y^s, X_1, ..., \nabla_X  X_r).
\end{eqnarray*}
the coordinate expression of which is denoted by
$$\nabla \alpha=(\alpha_{k_{1}\cdot\cdot\cdot
k_{r}; k_{r+1}}^{l_{1}\cdot\cdot\cdot l_{s}}).$$ We can continue to
define the second covariant derivative of $\alpha$ as follows:
\begin{eqnarray*}
&&\nabla^2 \alpha(Y^1, .., Y^s, X_1, ..., X_r, X, Y)
=(\nabla_{Y}(\nabla\alpha))(Y^1, .., Y^s, X_1, ..., X_r, X).
\end{eqnarray*}
the coordinate expression of which is denoted by
$$\nabla^2 \alpha=(\alpha_{k_{1}\cdot\cdot\cdot
k_{r}; k_{r+1}k_{r+2}}^{l_{1}\cdot\cdot\cdot l_{s}}).$$ Similarly,
we can also define the higher order covariant derivative of
$\alpha$:
$$\nabla^3 \alpha=\nabla(\nabla^2 \alpha), ... ,$$
and so on. For simplicity, the coordinate expression of the
covariant differentiation will usually be denoted by indices without
semicolons, e.g. $$u_{i}, \quad u_{ij} \quad \mbox{or} \quad
u_{ijk}$$ for a function $u: M\rightarrow \mathbb{R}$.

Our convention for the Riemannian curvature (3,1)-tensor Rm is
defined by
\begin{equation*}
Rm(X, Y)Z=-\nabla_{X}\nabla_{Y}Z+\nabla_{Y}\nabla_{X}Z+\nabla_{[X,
Y]}Z.
\end{equation*}
Pick a local coordinate chart $\{x^i\}_{i=1}^{n}$ of $M$. The
component of the (3,1)-tensor $Rm$ is defined by
\begin{equation*}
Rm\bigg({\frac{\partial}{\partial x^i}}, {\frac{\partial}{\partial
x^j}}\bigg){\frac{\partial}{\partial x^k}}=R_{ijk}^{\ \ \
l}{\frac{\partial}{\partial x^l}}
\end{equation*}
and $R_{ijkl}=g_{lm}R_{ijk}^{\ \ \ m}$. Then, we have the
standard commutation formulas (Ricci identities):
\begin{eqnarray*}\label{RI}
\alpha_{k_{1}\cdot\cdot\cdot k_{r};\ j i}^{l_{1}\cdot\cdot\cdot
l_{s}}-\alpha_{k_{1}\cdot\cdot\cdot k_{r};\ i
j}^{l_{1}\cdot\cdot\cdot l_{s}}=\sum_{a=1}^{r}R^{\ \ \ m}_{ijk_{l}}
\alpha_{k_{1}\cdot\cdot\cdot k_{a-1}m k_{a+1}\cdot\cdot\cdot
k_{r}}^{l_{1}\cdot\cdot\cdot l_{s}}-\sum_{b=1}^{s}R^{\ \ \
l_b}_{ijm} \alpha_{k_{1}\cdot\cdot\cdot k_{r}}^{l_{1}\cdot\cdot\cdot
l_{b-1}m l_{b+1}\cdot\cdot\cdot l_{r}}.
\end{eqnarray*}

Let $M$ be an immersed hypersurface in $\mathbb{R}^{n+1}$. Denote
$R_{ijkl}$ to be the Riemannian curvature of $M\subset
\mathbb{R}^{n+1}$ with the induced metric $g$. Pick a local
coordinate chart $\{x^i\}_{i=1}^{n}$ on $M$. Let $\nu$ be a given
unit normal and $h_{ij}$ be the second fundamental form $A$ of the
hypersurface with respect to $\nu$, that is
$$h_{ij}=-\langle\frac{\partial^2 X}{\partial x^i\partial x^j}, \nu\rangle_{\mathbb{R}^{n+1}}.$$
Recalling the following identities
\begin{equation}\label{Gauss for}
\nabla_i \nabla_j X=-h_{ij}\nu, \quad \quad \mbox{Gauss formula}
\end{equation}

\begin{equation}\label{Wein for}
\nabla_i \nu=h_{ij}X^j, \quad \quad \mbox{Weingarten formula}
\end{equation}

\begin{equation*}\label{Gauss}
R_{ijkl}=h_{ik}h_{jl}-h_{il}h_{jk}, \quad \quad \mbox{Gauss
equation}
\end{equation*}

\begin{equation}\label{Codazzi}
\nabla_{k}h_{ij}=\nabla_{j}h_{ik} \quad \quad \mbox{Codazzi,
equation}
\end{equation}
where $X^j=g^{ik}\nabla_i X$
Moreover, we have
\begin{eqnarray}\label{2rd}
\nabla_{i}\nabla_{j}h_{kl}
&=&\nabla_{k}\nabla_{l}h_{ij}+h^{m}_{j}(h_{il}h_{km}-h_{im}h_{kl})+h^{m}_{l}(h_{ij}h_{km}-h_{im}h_{kj}).
\end{eqnarray}

\subsection{Star-shaped hypersurfaces in $\mathbb{R}^{n+1}$}

Let $M$ be a star-shaped hypersurface in $\mathbb{R}^{n+1}$ which can represented by
\begin{eqnarray*}
M=\rho(x)x \quad \mbox{for} \quad x \in \mathbb{S}^n,
\end{eqnarray*}
where $X$ is the position vector of the hypersurface $M$ in $\mathbb{R}^{n+1}$.

Let $e_1,...,e_n$ be a smooth local orthonormal frame field on
$\mathbb{S}^n$ and $e_\rho$ be the radial vector field
in $\mathbb{R}^{n+1}$. $D_i\rho=D_{e_i} \rho$, $D_iD_j\rho=D^2
\rho(e_i, e_j)$ denote the covariant derivatives of $u$ with respect
to the round metric $\sigma$ of $\mathbb{S}^n$. Then, the following formulas
hold:

(i) The tangential vector on $M$ is
\begin{eqnarray*}
X_{i}=\rho e_{i}+D_i\rho e_{\rho}
\end{eqnarray*}
and the corresponding outward unit normal vector is given by
\begin{eqnarray}\label{Nor}
\nu=\frac{1}{v}\left(e_\rho-\frac{1}{\rho^2} D^j\rho e_j\right),
\end{eqnarray}
where $v=\sqrt{1+\rho^{-2}|D \rho|^2}$ with $D^j \rho=\sigma^{ij}D_i\rho$.

(ii) The induced metric $g$ on $M$ has the form
\begin{equation*}
g_{ij}=\rho^2\sigma_{ij}+D_i\rho D_j\rho
\end{equation*}
and its inverse is given by
\begin{equation*}
g^{ij}=\frac{1}{\rho^2}\left(\sigma^{ij}-\frac{D^i\rho D^j\rho}{\rho^2
v^{2}}\right).
\end{equation*}

(iii) The second fundamental form of $M$ is given by
\begin{eqnarray*}
h_{ij}=\frac{1}{v}\left(-D_iD_j\rho+\rho
\sigma_{ij}+\frac{2}{\rho}D_i\rho D_j\rho\right)
\end{eqnarray*}
and
\begin{eqnarray}\label{h_ij}
h^{i}_{j}=\frac{1}{\rho
v}\left(\delta^{i}_{j}+[-\sigma^{ik}+\frac{D^i\rho
D^k\rho}{\rho^2v}]D_jD_k(\log \rho)\right).
\end{eqnarray}

The following Newton-Maclaurin inequality will be used frequently.
\begin{lemma}\textit{(\cite{Tr90, LT94})} \label{lemma1}
\textit{\ Let} $\lambda\in\mathbb{R}^n$. \textit{ For }$0\leq
l<k\leq n,$ $r>s\ge0, k\ge r, l\ge s$, \textit{\ the following is
the Newton-Maclaurin inequality }

(1)
\[
k(n-l+1)\sigma_{l-1}(\lambda)\sigma_{k}(\lambda)\leq
l(n-k+1)\sigma_{l}(\lambda)\sigma_{k-1}(\lambda).
\]

(2)
\[\big[\frac{\sigma_k(\lambda)/C^k_n}{\sigma_l(\lambda)/C^l_n}\big]^\frac{1}{k-l}
\le\big[\frac{\sigma_r(\lambda)/C^r_n}{\sigma_s(\lambda)/C^s_n}\big]^{\frac{1}{r-s}},
\textit{for }  \lambda\in \Gamma_k.\]
\end{lemma}

To handle the ellipticity of the equation \eqref{Eq}, we need the
following important proposition and its proof is the same as
Proposition 2.2 in \cite{GZ19}.
\begin{proposition}\label{ellipticconcave}
Let $M$ be a smooth $k-1$ convex closed hypersurface in $\mathbb{R}^{n+1}$
and $\alpha_{l}(X)\geq 0$ for $\forall X \in M$ and $0\leq l\leq k-2$. Then the operator
\begin{eqnarray}
G(h_{ij}(X),
X)\allowdisplaybreaks\notag=\frac{\sigma_k(\kappa(X))}{\sigma_{k-1}
(\kappa(X))}
-\sum_{l=0}^{k-2}\alpha_{l}(X)\frac{\sigma_{l}(\kappa(X))}{\sigma_{k-1}(\kappa(X))}
\end{eqnarray}
is elliptic and concave about $h_{ij}(X)$.
\end{proposition}

\section{The a prior estimates}

In order to prove Theorem \ref{Main}, we consider the family of
equations as in \cite{An, Li-Sh} for $0\leq t\leq 1$
\begin{eqnarray}\label{Eq2}
\frac{\sigma_k(\kappa(X))}{\sigma_{k-1}
(\kappa(X))}-\sum_{l=0}^{k-2}t\alpha_{l}(X)\frac{\sigma_{l}(\kappa(X))}{\sigma_{k-1}(\kappa(X))}-\alpha_{k-1}(X,
t)=0,
\end{eqnarray}
where
\begin{eqnarray*}
\alpha_{k-1}(X,
t)=t\alpha_{k-1}(X)+(1-t)\phi(|X|)\frac{\sigma_k(e)}{\sigma_{k-1}
(e)}\frac{1}{|X|}
\end{eqnarray*}
and $\phi$ is a positive function which satisfies the following
conditions:

(a) $\phi(\rho)>0$,

(b) $\phi(\rho)>1$, for $\rho\leq r_1$,

(c) $\phi(\rho)<1$, for $\rho\geq r_2$,

(d) $\phi^{\prime}(\rho)<0$.

\subsection{$C^0$ Estimates}

Now, we can prove the following proposition which asserts all the
solution of the equation \eqref{Eq} have uniform $C^0$ bound.

\begin{proposition}\label{C^0}
Assume $0\leq\alpha_l(X) \in C^\infty(\mathbb{R}^{n+1})$.
Under the assumptions \eqref{ASS1} and \eqref{ASS2} mentioned in
Theorem \ref{Main} , if $M\subset \mathbb{R}^{n+1}$ is a
star-shaped, $k-1$-convex hypersurface satisfied the equation
\eqref{Eq2} for a given $t \in [0, 1]$, then
\begin{eqnarray*}
r_1<\rho(X)<r_2, \quad \forall \ X \in M.
\end{eqnarray*}
\end{proposition}

\begin{proof}
Assume $\rho(x)$ attains its maximum at $x_0 \in \mathbb{S}^n$ and
$\rho(x_0)\geq r_2$, then recalling \eqref{h_ij}
\begin{eqnarray*}
h^{i}_{j}=\frac{1}{uv}\left(\delta^{i}_{j}+[-\sigma^{im}+\frac{D^i\rho
D^m\rho}{\rho^2v}]D_jD_m(\log \rho)\right)
\end{eqnarray*}
which implies
\begin{eqnarray*}
h^{i}_{j}(x_0)=\frac{1}{\rho}[\delta^{i}_{j}-\sigma^{im}D_jD_m(\log
\rho)]\geq\frac{1}{\rho}\delta^{i}_{j}.
\end{eqnarray*}
Note that $\frac{\sigma_k}{\sigma_{k-1}}$ and $\frac{\sigma_{k-1}}{\sigma_{l}}$ for $0\leq l\leq k-2$ is concave in
$\Gamma_{k-1}$. Thus
\begin{eqnarray*}
\frac{\sigma_k}{\sigma_{k-1}}(h^{i}_{j}) \ge
\frac{\sigma_k}{\sigma_{k-1}}(\frac{1}{\rho}\delta^{i}_{j})
+\frac{\sigma_k}{\sigma_{k-1}}(-\frac{1}{\rho}\sigma^{im}D_jD_m(\log
\rho))
\ge\frac{\sigma_k}{\sigma_{k-1}}(\frac{1}{\rho}\delta^{i}_{j}).
\end{eqnarray*}
So,
\begin{eqnarray*}
\frac{\sigma_k(\kappa(X))}{\sigma_{k-1} (\kappa(X))}\geq
\frac{\sigma_k(e)}{\sigma_{k-1} (e)}\frac{1}{\rho}.
\end{eqnarray*}
Similarly,
\begin{eqnarray*}
\frac{\sigma_{l} (\kappa(X))}{\sigma_{k-1} (\kappa(X))}\leq
\frac{\sigma_{l} (e)}{\sigma_{k-1} (e)}\rho^{k+1-l}.
\end{eqnarray*}
Thus, we obtain by combining the above two inequalities
\begin{eqnarray*}
\frac{\sigma_k(e)}{\sigma_{k-1} (e)}\frac{1}{\rho}-
\sum_{l=0}^{k-2}t\alpha_{l}(X)\frac{\sigma_{l} (e)}{\sigma_{k-1}
(e)}\rho^{k-1-l}\leq\alpha_{k-1}(X, t)
\end{eqnarray*}
which is in contradiction with
\begin{eqnarray*}
\frac{\sigma_k(e)}{\sigma_{k-1} (e)}\frac{1}{\rho}-
\sum_{l=0}^{k-2}t\alpha_{l}(X)\frac{\sigma_{l} (e)}{\sigma_{k-1}
(e)}\rho^{k-1-l}>\alpha_{k-1}(X, t),
\end{eqnarray*}
in view of \eqref{ASS1} and the fact $\phi(\rho)<1$ for $\rho\geq
r_2$.
\end{proof}

Now, we prove the following uniqueness result.

\begin{proposition}\label{Uni}
For $t=0$, there exists an unique admissible solution of the
equation \eqref{Eq2}, namely $M=\{X: \rho(X)=\rho_0\}$, where
$\rho_0$ satisfies $\varphi(\rho_0)=1$.
\end{proposition}

\begin{proof}
Let $X$ be a solution of \eqref{Eq2} for $t=0$
\begin{eqnarray*}
\frac{\sigma_k(\kappa(X))}{\sigma_{k-1} (\kappa(X))}
-\phi(|X|)\frac{\sigma_k(e)}{\sigma_{k-1} (e)}\frac{1}{|X|}=0.
\end{eqnarray*}
Assume $\rho(x)$ attains its maximum $\rho_{max}$ at $x_0 \in
\mathbb{S}^n$, then
\begin{eqnarray*}
\frac{\sigma_k(\kappa(X))}{\sigma_{k-1} (\kappa(X))}\geq
\frac{\sigma_k(e)}{\sigma_{k-1} (e)}\frac{1}{\rho},
\end{eqnarray*}
which implies
\begin{eqnarray*}
\varphi(\rho_{max})\geq 1.
\end{eqnarray*}
Similarly,
\begin{eqnarray*}
\varphi(\rho_{min})\leq 1.
\end{eqnarray*}
Thus, since $\varphi$ is a decreasing function, we have
\begin{eqnarray*}
\varphi(\rho_{min})=\varphi(\rho_{max})=1.
\end{eqnarray*}
We conclude
\begin{eqnarray*}
\rho(X)=\rho_0 \quad \mbox{for} \quad X \in M,
\end{eqnarray*}
where $\rho_0$ is the unique solution of $\varphi(\rho_0)=1$.
\end{proof}

\subsection{$C^1$ Estimates}

In this section, we establish the gradient estimate for the
equation. The treatment of this section follows largely from
\cite{Ca, Guan-Ren15}. We can rewritten the equation \eqref{Eq2} as:
\begin{eqnarray}\label{Eq1}
G(h_{ij}(X), X, t)\allowdisplaybreaks\notag
=\frac{\sigma_k(\kappa(X))}{\sigma_{k-1} (\kappa(X))}
-\sum_{l=0}^{k-2}t\alpha_{l}(X)\frac{\sigma_{l}(\kappa(X))}{\sigma_{k-1}(\kappa(X))}=\alpha_{k-1}(X,
t)
\end{eqnarray}
For the convenience of notations, we will denote
\begin{eqnarray*}
G_k(h_{ij}(X))=\frac{\sigma_k(\kappa(X))}{\sigma_{k-1}
(\kappa(X))},   \quad
G_l(h_{ij}(X))=-\frac{\sigma_{l}(\kappa(X))}{\sigma_{k-1}(\kappa(X))},
\end{eqnarray*}
and
\begin{eqnarray*}
G^{ij}(\kappa(X))=\frac{\partial G}{\partial h_{ij}}, \quad G^{ij,
r s}(\kappa(X))=\frac{\partial^2 G}{\partial h_{ij}h_{rs}}.
\end{eqnarray*}

Recalling that a star-shaped hypersurface $M$ in $\mathbb{R}^{n+1}$ can represented by
\begin{eqnarray*}
X(x)=\rho(x)x \quad \mbox{for} \quad x \in \mathbb{S}^n,
\end{eqnarray*}
where $X$ is the position vector of the hypersurface $M$ in $\mathbb{R}^{n+1}$.
We can get the following gradient bound.

\begin{proposition}\label{C^1}
Under the assumption \eqref{ASS3}, if the $k-1$ convex, star-shaped hypersurface
$M$ satisfies \eqref{Eq} and $\rho$ has positive upper and lower
bound, then there exists a constant $C$ depending on the minimum and
maximum values of $\rho$ such that
\begin{eqnarray*}
|D \rho(x)|\leq C, \quad \mbox{for} \quad \forall X \in \mathbb{S}^n.
\end{eqnarray*}
\end{proposition}

\begin{proof}
First, we know from \eqref{Nor}
$$\langle X, \nu\rangle=\frac{\rho^2}{\sqrt{\rho^{2}+|D \rho|^2}},$$
it is sufficient to obtain a positive lower bound of $\langle X,
\nu\rangle$. We consider
\begin{eqnarray*}
\phi=-\log \langle X, \nu\rangle+\gamma(|X|^2),
\end{eqnarray*}
where $\gamma(t)$ is a function which will be chosen later.
Assume $X_0$ is the maximum value point of $\phi$. If $X$ is
parallel to the normal direction $\nu$ of at $X_0$, we have
$\langle X, \nu\rangle=|X|$. Thus, our result holds true.
So, we assume $X$ is not
parallel to the normal direction $\nu$ at $X_0$, we may choose the
local orthonormal frame $\{e_1,..., e_n\}$ on $M$ satisfying
\begin{eqnarray*}
\langle X, e_1\rangle\neq 0, \quad \mbox{and} \quad \langle X,
e_i\rangle=0, \quad i\geq 2.
\end{eqnarray*}
Then, we arrive at $X_0$,
\begin{eqnarray}\label{Par-1}
\langle X, \nu\rangle_i=2\langle X, \nu\rangle\gamma^{\prime}\langle
X, e_i\rangle
\end{eqnarray}
and
\begin{eqnarray*}
\phi_{ii}&=&-\frac{\langle X, \nu\rangle_{ii}}{\langle X,
\nu\rangle}+\frac{(\langle X, \nu\rangle_{i})^2}{\langle X,
\nu\rangle^2}+\gamma^{\prime
\prime}(|X|^{2}_{i})^2+\gamma^{\prime}|X|^2_{ii}
\\&=&-\frac{1}{\langle X, \nu\rangle}[h_{ii; 1}\langle X, e_1\rangle+h_{ii}-h^{2}_{ii}\langle X, \nu\rangle]
\\&&+4[(\gamma^{\prime})^2+\gamma^{\prime \prime}]\langle X,
e_1\rangle^2\delta_{1i}+\gamma^{\prime}[2-2\langle X, \nu\rangle
h_{ii}]
\end{eqnarray*}
in view of
\begin{eqnarray*}
\langle X, \nu\rangle_{ii}=h_{ii}+h_{ii; l}\langle X,
e_l\rangle-h_{im}h^{m}_{i}\langle X, \nu\rangle.
\end{eqnarray*}
By Weingarten formula \eqref{Wein for} and \eqref{Par-1}, we have at
$X_0$
\begin{eqnarray}\label{Par-1-1}
h_{11}=2\langle X, \nu\rangle\gamma^{\prime}, \quad h_{1i}=0, \
i\geq 2.
\end{eqnarray}
Therefore, we can rotate the coordinate system such that
$\{e_i\}_{i=1}^{n}$ are the principal curvature directions of the
second fundamental form $(h_{ij})$, i.e. $h_{ij}=h_{ii}\delta_{ij}$
Thus,
\begin{eqnarray*}
G^{ii}\phi_{ii}&=&-\frac{\langle X, e_1\rangle}{\langle X,
\nu\rangle}G^{ii}h_{ii; 1}-\frac{1}{\langle X,
\nu\rangle}G^{ii}h_{ii}+G^{ii}h^{2}_{ii}
\\&&+4[(\gamma^{\prime})^2+\gamma^{\prime \prime}]G^{11}\langle X, e_1\rangle^2+\gamma^{\prime}G^{ii}[2-2\langle X, \nu\rangle h_{ii}].
\end{eqnarray*}
Note that
$$G^{ij}h_{ij}=G-\sum_{l=0}^{k-2}(k-l)t\alpha_l G_l=\alpha_{k-1}(X, t)-\sum_{l=0}^{k-2}(k-l)t\alpha_l G_l$$
and
$$G^{ij}h_{ij; 1}=\nabla_1\alpha_{k-1}(X, t)-\sum_{l=0}^{k-2}t\nabla_1\alpha_l G_l.$$
We conclude
\begin{eqnarray}\label{C1-3}
G^{ii}\phi_{ii}&=&\frac{\langle X, e_1\rangle}{\langle X,
\nu\rangle}\bigg(-\nabla_1\alpha_{k-1}(X,
t)+\sum_{l=0}^{k-2}t\nabla_1\alpha_l G_l\bigg)\\
\nonumber&&+\frac{1}{\langle X, \nu\rangle}\bigg(-\alpha_{k-1}(X,
t)+\sum_{l=0}^{k-2}(k-l)t\alpha_l G_l\bigg)\\
\nonumber&&+G^{ii}h^{2}_{ii}+4[(\gamma^{\prime})^2+\gamma^{\prime
\prime}]G^{11}\langle X,
e_1\rangle^2+\gamma^{\prime}G^{ii}[2-2\langle X, \nu\rangle h_{ii}]
\\ \nonumber&=&\frac{1}{\langle X, \nu\rangle}
\bigg(-\langle X, e_1\rangle\nabla_1 \alpha_{k-1}(X,
t)-\alpha_{k-1}(X, t)\bigg)
\\ \nonumber&&+\frac{1}{\langle X, \nu\rangle}\sum_{l=0}^{k-2}tG_l
\bigg(\langle X, e_1\rangle\nabla_1
\alpha_{l}+(k-l)\alpha_{l}\bigg)+G^{ii}h^{2}_{ii} \nonumber\\
\nonumber&&+4[(\gamma^{\prime})^2+\gamma^{\prime
\prime}]G^{11}\langle X, e_1\rangle^2+\gamma^{\prime}G^{ii}
[2-2\langle X, \nu\rangle h_{ii}].
\end{eqnarray}
Since $\langle X, e_i\rangle=0$ for $i=2, ..., n$, $$X=\langle X,
e_1\rangle e_1+\langle X, \nu\rangle \nu$$ which results in
\begin{eqnarray*}
\langle X, e_1\rangle\nabla_1\alpha_{l}(X)+(k-l)\alpha_{l}(X)=X
d_X\alpha_{l}(X)+(k-l)\alpha_{l}(X)-\langle X, \nu\rangle
d_{\nu}\alpha_{l}(X).
\end{eqnarray*}
We know from the assumption \eqref{ASS3}
\begin{eqnarray*}
[(k-l)\alpha_l+Xd_X\alpha_l(X)]=[(k-l)\alpha_l+\rho\frac{\partial
\alpha_l(X)}{\partial \rho}]\leq 0.
\end{eqnarray*}
Thus,
\begin{eqnarray}\label{C1-1}
\langle X,
e_1\rangle\nabla_1\alpha_{l}(X)+(k-l)\alpha_{l}(X)\leq-\langle X,
\nu\rangle d_{\nu}\alpha_{l}(X).
\end{eqnarray}
and
\begin{eqnarray}\label{C1-11}
&&\langle X, e_1\rangle\nabla_1\alpha_{k-1}(X, t)+\alpha_{k-1}(X,
t)\\ \nonumber&\leq&
(1-t)\varphi^{\prime}\frac{\sigma_k(e)}{\sigma_{k-1}(e)}-\langle X,
\nu\rangle d_{\nu}\alpha_{k-1}(X, t).
\end{eqnarray}
Taking \eqref{C1-1} and \eqref{C1-11} into \eqref{C1-3}, we have at
$x_0$
\begin{eqnarray*}
0&\geq& G^{ii}\phi_{ii}\\&\geq&G^{ii}h^{2}_{ii}
+4[(\gamma^{\prime})^2+\gamma^{\prime \prime}]G^{11}\langle X,
e_1\rangle^2+\gamma^{\prime}G^{ii}[2-2\langle X, \nu\rangle
h_{ii}]\\&&+t\sum_{l=0}^{k-2}G_l
d_{\nu}\alpha_{l}(X)-\frac{(1-t)}{\langle X,
\nu\rangle}\varphi^{\prime}\frac{\sigma_k(e)}{\sigma_{k-1}(e)}+
d_{\nu}\alpha_{k-1}(X, t)
\\&=&G^{ii}[h_{ii}-\gamma^{\prime}\langle X, \nu\rangle]^2
+4[(\gamma^{\prime})^2+\gamma^{\prime \prime}]G^{11}\langle X,
e_1\rangle^2+G^{ii}[2\gamma^{\prime}-(\gamma^{\prime})^2\langle X,
\nu\rangle^2]\\&&-t\sum_{l=0}^{k-2}G_l
d_{\nu}\alpha_{l}(X)-\frac{(1-t)}{\langle X,
\nu\rangle}\varphi^{\prime}\frac{\sigma_k(e)}{\sigma_{k-1}(e)}+
d_{\nu}\alpha_{k-1}(X, t).
\end{eqnarray*}
Choosing $$\gamma(t)=\frac{\alpha}{t}$$ for sufficiently large
$\alpha$. Therefore,
\begin{eqnarray*}
0\geq G^{ii}[2\gamma^{\prime}-(\gamma^{\prime})^2\langle X,
\nu\rangle^2]-C(\sum_{l=0}^{k-2}|G_l|+1).
\end{eqnarray*}
in view of
\begin{eqnarray*}
(\gamma^{\prime})^2+\gamma^{\prime \prime}\geq 0.
\end{eqnarray*}
To continue our proof, we need to estimate $G_l$ for $0\leq l\leq
k-2$. Let $N \in \mathbb{R}^1$ be a fixed positive number.

 (1) If $\frac{\sigma_k}{\sigma_{k-1}}\le N$, then we get from
$\alpha_l(X)\geq c_l$
\[|G_l|=\frac{\sigma_{l}}{\sigma_{k-1}}\leq\frac{1}{\alpha_l}
\big(\frac{\sigma_k}{\sigma_{k-1}}+\alpha_{l-1}(X, t)\big)\le
C(N+1).\]

(2) If $\frac{\sigma_k}{\sigma_{k-1}}> N$, then by Lemma
\ref{lemma1},
\begin{equation*}
|G_l|=\frac{\sigma_l}{\sigma_{k-1}}=\frac{\sigma_l}{\sigma_{l+1}}
\cdot\frac{\sigma_{l+1}}{\sigma_{l+2}}\cdot\cdot\cdot\cdot\frac{\sigma_{k-2}}{\sigma_{k-1}}\le
C\big(\frac{\sigma_{k-1}}{\sigma_k}\big)^{k-1-l}\le N^{-(k-1-l)}.
\end{equation*}
So, $|G_l|$ are bounded. By the definition of operator $G$ and
straightforward computation, we have $\sum_i
G^{ii}\ge\frac{n-k+1}{k}$ (See \cite{GZ19}), so we can choose
sufficiently large $\alpha$ such that
\begin{eqnarray*}
0\geq G^{ii}[\gamma^{\prime}-(\gamma^{\prime})^2\langle X,
\nu\rangle^2]
\end{eqnarray*}
Thus,
\begin{eqnarray*}
\gamma^{\prime}\leq(\gamma^{\prime})^2\langle X, \nu\rangle^2,
\end{eqnarray*}
which means
\begin{eqnarray*}
\langle X, \nu\rangle(X_0)\geq C.
\end{eqnarray*}
So, our proof is complete.
\end{proof}

\subsection{$C^2$ Estimates}
In order to get the estimate of the second fundamental form, we
first need some lemmas.
\begin{lemma}\label{C^2-1}
Let $M$ be a $(k-1)$ convex solution of \eqref{Eq2} with the
position vector $X$ in $\mathbb{R}^{n+1}$ and assume that
$\alpha_l(X)\geq 0$ for $0\leq l\leq k-1$ and $X \in M$. Then, we
have the following inequality
\begin{eqnarray*}
G^{ij}h_{ij; pp}\geq\nabla_p\nabla_p\alpha_{k-1}(X,
t)-\sum_{l=0}^{k-2}\frac{1}{1+\frac{1}{k+1-l}}\frac{t(\nabla_p\alpha_l)^2}{\alpha_l}G_l-\sum_{l=0}^{k-2}t\nabla_p\nabla_p\alpha_{l}G_{l}.
\end{eqnarray*}
\end{lemma}

\begin{proof}
 Differentiating the equation \eqref{Eq2} once, we can have
\begin{eqnarray*}
\nabla_p\alpha_{k-1}(X, t)=G^{ij}h_{ij; p}
+\sum_{l=0}^{k-2}t\nabla_p \alpha_{l}G_{l}.
\end{eqnarray*}
Differentiating the equation \eqref{Eq2} twice, we can have
\begin{eqnarray*}
\nabla_p\nabla_p\alpha_{k-1}(X, t)=\sum_{l=1}^{n}G^{ij,rs}h_{ij;
p}h_{rs; p} +G^{ij}h_{ij;
pp}+2\sum_{l=0}^{k-2}t\nabla_p\alpha_{l}G_{l}^{ij}h_{ij;
p}+\sum_{l=0}^{k-2}t\nabla_p\nabla_p\alpha_{l}G_{l}.
\end{eqnarray*}
Moreover, since the operator $(\frac{\sigma_{k-1}}{\sigma_l})^{\frac{1}{k-1-l}}$
is concave for $0\leq l\leq k-2$, we have from (see also (3.10) in \cite{GZ19})
\begin{eqnarray}
-G_l^{ij,rs}h_{ij; p}h_{rs; p} \ge
-\big(1+\frac{1}{k-1-l}\big)G_l^{-1}G_l^{ij}G_l^{rs}h_{ij; p}h_{rs;
p}.\label{93}
\end{eqnarray}
Thus, we have in view that $G_k$ is concave in $\Gamma_{k-1}$
\begin{eqnarray*}
&&\nabla_p\nabla_p\alpha_{k-1}(X,
t)\\&\leq&\sum_{l=1}^{k-2}t\alpha_lG_{l}^{ij,rs}h_{ij; p}h_{rs;
p}+G^{ij}h_{ij;
pp}+2\sum_{l=0}^{k-2}t\nabla_p\alpha_{l}G_{l}^{ij}h_{ij;
p}+\sum_{l=0}^{k-2}t\nabla_p\nabla_p\alpha_{l}G_{l}\\&\leq&
\sum_{l=1}^{k-2}t\alpha_l\big(1+\frac{1}{k-1-l}\big)G_l^{-1}(G_l^{ij}h_{ij;
p})^2+G^{ij}h_{ij; pp}
+2\sum_{l=0}^{k-2}t\nabla_p\alpha_{l}G_{l}^{ij}h_{ij;
p}\\&&+\sum_{l=0}^{k-2}t\nabla_p\nabla_p\alpha_{l}G_{l}
\\&=&\frac{k-l}{k-1-l}
\sum_{l=1}^{k-2}t\alpha_lG_l^{-1}\bigg(G_l^{ij}h_{ij;
p}+\frac{1}{1+\frac{1}{k-1-l}}\frac{\nabla_p\alpha_l}{\alpha_l}G_l\bigg)^2+
\sum_{l=0}^{k-2}\frac{1}{1+\frac{1}{k-1-l}}\frac{t(\nabla_p\alpha_l)^2}{\alpha_l}G_l\\&&+G^{ij}h_{ij;
pp}+\sum_{l=0}^{k-2}t\nabla_p\nabla_p\alpha_{l}G_{l}\\&\leq&
\sum_{l=0}^{k-2}\frac{1}{1+\frac{1}{k-1-l}}\frac{t(\nabla_p\alpha_l)^2}{\alpha_l}G_l+G^{ij}h_{ij;
pp}+\sum_{l=0}^{k-2}t\nabla_p\nabla_p\alpha_{l}G_{l}.
\end{eqnarray*}
So, our proof is complete.
\end{proof}

\begin{lemma}\label{C^2-2}
Let $M$ be a $(k-1)$ convex solution of \eqref{Eq2} with the
position vector $X$ in $\mathbb{R}^{n+1}$. We have the following
equality
\begin{eqnarray*}
&&G^{ij}\langle X, \nu\rangle_{ij}+\langle X, \nu\rangle
G^{ij}h_{im}h^{m}_{j}\\&=&\bigg(\nabla_p\alpha_{k-1}(X,
t)-\sum_{l=0}^{k-2}t\nabla_p\alpha_l G_l\bigg)\langle X,
X^p\rangle+\alpha_{k-1}(X, t)-\sum_{l=0}^{k-2}(k-l)t\alpha_l G_l.
\end{eqnarray*}
\end{lemma}

\begin{proof}
We have by Gauss formula \eqref{Gauss for} and Codazzi equation
\eqref{Codazzi}
\begin{eqnarray*}
\langle X, \nu\rangle_{ij}&=&\langle X_{ij}, \nu\rangle+2\langle
X_i, \nu_j\rangle+\langle X, \nu_{ij}\rangle
\\&=&-h_{ij}+2h_{ij}+h_{ij; p}\langle X^{p}, X\rangle-h_{im}h^{m}_{j}\langle X, \nu\rangle
\\&=&h_{ij}+h_{ij; p}\langle X^{p}, X\rangle-h_{im}h^{m}_{j}\langle X, \nu\rangle,
\end{eqnarray*}
which results in
\begin{eqnarray*}
G^{ij}\langle X, \nu\rangle_{ij}=G^{ij}h_{ij; p}\langle X,
X^p\rangle+G^{ij}h_{ij}-\langle X,\nu\rangle G^{ij}h_{im}h^{m}_{j}.
\end{eqnarray*}
Note that
$$G^{ij}h_{ij}=G-\sum_{l=0}^{k-2}(k-l)t\alpha_l G_l=\alpha_{k-1}(X, t)-\sum_{l=0}^{k-2}(k-l)t\alpha_l G_l$$
and
$$G^{ij}h_{ij; p}=\nabla_p\alpha_{k-1}(X, t)-\sum_{l=0}^{k-2}t\nabla_p\alpha_l G_l.$$
Thus,
\begin{eqnarray*}
G^{ij}\langle X, \nu\rangle_{ij}&=&\bigg(\nabla_p\alpha_{k-1}(X,
t)-\sum_{l=0}^{k-2}t\nabla_p\alpha_l G_l\bigg)\langle X,
X^p\rangle+\alpha_{k-1}(X, t)\\&&-\sum_{l=0}^{k-2}(k-l)t\alpha_l
G_l-\langle X, \nu\rangle G^{ij}h_{im}h^{m}_{j}.
\end{eqnarray*}
\end{proof}
So, we complete our proof.

Now we begin to estimate the second fundamental form.

\begin{proposition}\label{C^2}
If the $k$-convex hypersurface $M$ satisfies \eqref{Eq2} with the
position vector $X$ in $\mathbb{R}^{n+1}$. Assume that $k\geq 2$ and
$$\alpha_l(X)\geq c_l>0, \quad \forall X \in M,$$ for $0\leq l\leq
k-1$. Then there exists a constant $C$, depending on $k, n$, $c_l$,
$|X|_{C^2}$, $|\alpha_l|_{C^2}$ and $|\nabla \rho|_{C^0}$ such that
for $1\leq i\leq n$
\begin{equation*}
|\kappa_{i}(X)|\le C, \quad \mbox{for} \quad X \in M
\end{equation*}
\end{proposition}

\begin{proof}
Since $k\geq 2$, $M$ is 2-convex. Thus,
\begin{equation*}
|\kappa_i|\leq C H.
\end{equation*}
So, we only need to estimate the mean curvature $H$ of $M$. Taking
the auxillary function
$$W(X)=\log H-\log\langle X, \nu\rangle.$$ Assume that $X_0$ is the
maximum point of $W$. Then at $X_0,$
\begin{equation}\label{C2-1}
0=W_{i}=\frac{H_i}{H}-\frac{\langle X, \nu\rangle_i}{\langle X, \nu\rangle}
\end{equation}
and
\begin{equation}\label{C2-2}
0\geq W_{ij}(x_0)=\frac{H_{ij}}{H}-\frac{\langle X, \nu\rangle_{ij}}{\langle X, \nu\rangle}.
\end{equation}
Choosing a suitable coordinate $\{x^1, x^2, ..., x^n\}$ on the
neighborhood of $X_0 \in M$ such that the matrix $\{h_{ij}\}$ is
diagonal at $X_0$. This implies at $x_0$
\begin{equation*}
0\geq G^{ij}W_{ij}(x_0)=\sum_{p=1}^{n}\frac{1}{H}G^{ii}h_{pp;
ii}-\frac{G^{ii}\langle X, \nu\rangle_{ii}}{\langle X, \nu\rangle}.
\end{equation*}
From \eqref{2rd}, we obtain
\begin{eqnarray*}
h_{pp; ii}=h_{ii; pp}-h^{m}_{i}h_{im}h_{pp}+|A|^2h_{ii},
\end{eqnarray*}
which results in at $x_0$ from  Lemma \ref{C^2-1}
\begin{eqnarray*}
0&\geq&\frac{1}{H}\sum_{p=1}^{n}G^{ii}h_{ii;
pp}-G^{ii}h_{mi}h^{m}_{i}+\frac{|A|^2}{H}\bigg(\alpha_{k-1}(X,
t)-\sum_{l=0}^{k-2}(k-l)t\alpha_l G_l\bigg)-\frac{G^{ii}\langle X,
\nu\rangle_{ii}}{\langle X, \nu\rangle}\\&\geq&
\frac{1}{H}\sum_{p=1}^{n}\bigg(\nabla_p\nabla_p\alpha_{k-1}(X,
t)-\sum_{l=0}^{k-2}\frac{1}{1+\frac{1}{k-1-l}}\frac{t(\nabla_p\alpha_l)^2}{\alpha_l}G_l
-\sum_{l=0}^{k-2}t\nabla_p\nabla_p\alpha_{l}G_{l}\bigg)-G^{ii}h_{mi}h^{m}_{i}\\&&+\frac{|A|^2}{H}\bigg(\alpha_{k-1}(X,
t)-\sum_{l=0}^{k-2}(k-l)t\alpha_l G_l\bigg)-\frac{G^{ii}\langle X,
\nu\rangle_{ii}}{\langle X, \nu\rangle}\\&\geq&
\frac{1}{H}\sum_{p=1}^{n}\bigg(\nabla_p\nabla_p\alpha_{k-1}(X,
t)-\sum_{l=0}^{k-2}\frac{1}{1+\frac{1}{k-1-l}}\frac{t(\nabla_p\alpha_l)^2}{\alpha_l}G_l
-\sum_{l=0}^{k-2}t\nabla_p\nabla_p\alpha_{l}G_{l}\bigg)\\&&+\frac{|A|^2}{H}\bigg(\alpha_{k-1}(X,
t)-\sum_{l=0}^{k-2}(k-l)t\alpha_l G_l\bigg)\\&&-\frac{1}{\langle X,
\nu\rangle}\bigg(\nabla_p\alpha_{k-1}(X,
t)-\sum_{l=0}^{k-2}t\nabla_p\alpha_l G_l\bigg)\langle X,
X^p\rangle-\frac{1}{\langle X, \nu\rangle}\alpha_{k-1}(X,
t)\\&&+\frac{1}{\langle X, \nu\rangle}\sum_{l=0}^{k-2}(k-l)t\alpha_l
G_l,
\end{eqnarray*}
where we use Lemma \ref{C^2-2} to get the last inequality.
Note that
\begin{eqnarray*}
\nabla_p\alpha_l=\frac{\partial \alpha_l}{\partial X^i}X^{i}_{p},
\quad
\sum_{p=1}^{n}\nabla_p\nabla_p\alpha_{l}=\sum_{p=1}^{n}\frac{\partial
\alpha_l}{\partial X^i\partial
X^j}X^{i}_{p}X^{j}_{p}-H\frac{\partial \alpha_l}{\partial
X^i}\nu^{i}
\end{eqnarray*}
in view of Weingarten formula \eqref{Wein for}. Then, we have
together with the fact $|A|^2\geq \frac{1}{n}H^2$
\begin{eqnarray*}
C\frac{1}{H}\bigg(\sum_{l=0}^{k-2}|G_l|+1\bigg)(H+1)
+C\bigg(\sum_{l=0}^{k-2}|G_l|+1\bigg)\geq\frac{H}{n}\alpha_{k-1}(X,
t)\geq C H.
\end{eqnarray*}
Let us divide the proof into two cases.

(1) $\frac{\sigma_k}{\sigma_{k-1}}\le H^{\frac{1}{k}}$. Then, we get
from $\alpha_l(X)\geq c_l$
\[|G_l|=\frac{\sigma_l}{\sigma_{k-1}}\leq\frac{1}{\alpha_l}\big(\frac{\sigma_k}{\sigma_{k-1}}+\alpha_{k-1}(X, t)\big)\le CH^{\frac{1}{k}},\]
which implies $H \le C$.

(2) $\frac{\sigma_k}{\sigma_{k-1}}> H^{\frac{1}{k}}$.
Then by Lemma \ref{lemma1},
\begin{equation*}
|G_l|=\frac{\sigma_l}{\sigma_{k-1}}=\frac{\sigma_l}{\sigma_{l+1}}
\cdot\frac{\sigma_{l+1}}{\sigma_{l+2}}\cdot\cdot\cdot\cdot\frac{\sigma_{k-2}}{\sigma_{k-1}}\le
C\big(\frac{\sigma_{k-1}}{\sigma_k}\big)^{k-1-l}\le
H^{-\frac{k-1-l}{k}}.
\end{equation*}
Now we also derive $H\le C$. So, our proof is complete.
\end{proof}

\section{The proof of the Theorem}

In this section, we use the degree theory for nonlinear elliptic
equation developed in \cite{Li89} to prove Theorem \ref{Main}. The
proof here is similar to \cite{An, Jin, Li-Sh}. So, only sketch will
be given below.

After establishing the  a priori estimates Proposition \ref{C^0},
Proposition \ref{C^1} and Proposition \ref{C^2}, we know that the
equation \eqref{Eq} is uniformly elliptic. From \cite{Eva82},
\cite{Kry83}, and Schauder estimates, we have
\begin{eqnarray}\label{C2+}
|\rho|_{C^{4,\alpha}(\mathbb{S}^n)}\leq C
\end{eqnarray}
for any $k$-convex solution $M$ to the equation \eqref{Eq}, where
the position vector of $M$ is $X=\rho(x)x$ for $x \in \mathbb{S}^n$.
We define
\begin{eqnarray*}
C_{0}^{4,\alpha}(\mathbb{S}^n)=\{\rho \in
C^{4,\alpha}(\mathbb{S}^n): M \ \mbox{is}
 \ k-\mbox{convex}\}.
\end{eqnarray*}
Let us consider $$F(.; t): C_{0}^{4,\alpha}(\mathbb{S}^n)\rightarrow
C^{2,\alpha}(\mathbb{S}^n)$$ which is defined by
\begin{eqnarray*}
F(\rho, x; t)=\frac{\sigma_k(\kappa(X))}{\sigma_{k-1}
(\kappa(X))}-\sum_{l=0}^{k-2}t\alpha_{l}(X)\frac{\sigma_{l}
(\kappa(X))}{\sigma_{k-1}(\kappa(X))}-\alpha_{k-1}(X, t).
\end{eqnarray*}
Let $$\mathcal{O}_R=\{\rho \in C_{0}^{4,\alpha}(\mathbb{S}^n):
|\rho|_{C^{4,\alpha}(\mathbb{S}^n)}<R\}$$ which clearly is an open
set of $C_{0}^{4,\alpha}(\mathbb{S}^n)$. Moreover, if $R$ is
sufficiently large, $F(\rho, x; t)=0$ has no solution on $\partial
\mathcal{O}_R$ by the a prior estimate established in \eqref{C2+}.
Therefore the degree $\deg(F(.; t), \mathcal{O}_R, 0)$ is
well-defined for $0\leq t\leq 1$. Using the homotopic invariance of
the degree, we have
\begin{eqnarray*}
\deg(F(.; 1), \mathcal{O}_R, 0)=\deg(F(.; 0), \mathcal{O}_R, 0).
\end{eqnarray*}
Proposition \ref{Uni} shows that $\rho=\rho_0$ is the unique
solution to the above equation for $t=0$. Direct calculation show
that
\begin{eqnarray*}
F(s\rho_0, x; 0)=[1-\varphi(s\rho_0)]\frac{\sigma_k(e)}{\sigma_{k-1}
(e)}\frac{1}{s\rho_0}.
\end{eqnarray*}
Using the fact $\varphi(\rho_0)=1$, we have
\begin{eqnarray*}
\delta_{\rho_0}F(\rho_0, x; 0)=\frac{d}{d s}|_{s=1}F(s\rho_0, x;
0)=-\varphi^{\prime}(\rho_0)\frac{\sigma_k(e)}{\sigma_{k-1} (e)}>0,
\end{eqnarray*}
where $\delta F(\rho_0, x; 0)$ is the linearized operator of $F$ at
$\rho_0$. Clearly, $\delta F(\rho_0, x; 0)$ takes the form
\begin{eqnarray*}
\delta_{w}F(\rho_0, x; 0)=-a^{ij}w_{ij}+b^i
w_i-\varphi^{\prime}(\rho_0)\frac{\sigma_k(e)}{\sigma_{k-1} (e)} w,
\end{eqnarray*}
where $a^{ij}$ is a positive definite matrix. Since
$-\varphi^{\prime}(\rho_0)\frac{\sigma_k(e)}{\sigma_{k-1} (e)}>0,$
thus $\delta F(\rho_0, x; 0)$ is an invertible operator. Therefore,
\begin{eqnarray*}
\deg(F(.; 1), \mathcal{O}_R; 0)=\deg(F(.; 0), \mathcal{O}_R, 0)=\pm
1.
\end{eqnarray*}
So, we obtain a solution at $t=1$. This completes the proof of
Theorem \ref{Main}.


\bigskip

\bigskip

\begin{thebibliography}{99}

\bibitem{Al} A. D. Aleksandrov, Uniqueness theorems for surfaces in the large (Russian),
Vestnik Leningrad. Univ., vol. 11
(1956), no. 19, 5-17; vol. 12 (1957), no. 7, 15-44; vol. 13 (1958),
no. 7, 14-26; vol. 13 (1958),no. 13, 27-34; vol. 13 (1958), no.
19,5-8; vol. 14(1959), no. 1, 5-13; vol. 15 (1960), no. 7, 5-13.

\bibitem{An} F.J. Andrade, J.L. Barbosa, J.H. de Lira, Closed Weingarten hypersurfaces
in warped product manifolds, Indiana
Univ. Math. J. 58(4), 1691-1718(2009).

\bibitem{Ba1} I. Ja. Bakelman and B. E. Kantor, Estimates of the solutions of
quasilinear elliptic equations that are connected with problems of
geometry in the large, (Russian) Mat. Sb. (N.S.) 91(133) (1973),
336-349, 471. [5]

\bibitem{Ba2} I. Ja. Bakelman and B. E. Kantor, Existence of a hypersurface
homeomorphic to the sphere in Euclidean space with a given mean
curvature, Geometry and topology, 1 (Russian), 3-10, Gos. Ped. Inst.
im. Gercena, Leningrad, 1974.

\bibitem{Ba-Li}  J. L. Barbosa, J. H. de Lira and V. I. Oliker, A priori estimates
for starshaped compacthypersurfaces with prescribed
mth curvature function in space forms, Nonlinear problems in
mathematical physics and related topics, I, 35-52, Int. Math. Ser.
(N. Y.), 1, Kluwer/Plenum, New York, 2002.

\bibitem{Ca} L. Caffarelli, L. Nirenberg and J. Spruck, Nonlinear second order
elliptic equations, IV. Starshaped compact Wein-
gartenhypersurfaces,CurrentTopics in PDE¡¯s, edited byY. Ohya,K.
Kosahara,N. Shimakura,KinokuniaCompany LTD, Tokyo 1986, 1-26.

\bibitem{Chen} Daguang Chen, Haizhong Li, Zhizhang Wang. Starshaped compact hypersurfaces with prescirbed Weingarten
curvature in warped product manifolds. arXiv:1705.00313

\bibitem{Chen19} Li Chen, Xi Guo and Yan He:
A class of fully nonlinear equations arising in conformal geometry, preprint.

\bibitem{Co17} T. Collins and G. Sz$\acute{e}$kelyhidi, Convergence of the
$J$-flow on toric manifolds, J. Differential Geom. 107
(2017) no. 1, 47-81

\bibitem{Eva82} L. C. Evans:
Classical solutions of fully nonlinear, convex, second-order elliptic equations,
Comm. Pure Appl. Math. 35, no. 3(1982), 333--363.

\bibitem{Fi} W.J. Firey, Christoffel problem for general convex bodies, Mathematik 15, 7-21 (1968)

\bibitem{Fu08} J.X. Fu and S.-T. Yau, The theory of superstring with flux on non-K¡§ahler manifolds and the complex
Monge-Amp`ere equation, J. Differential Geom., Vol 78, Number 3 (2008), 369-428.

\bibitem{Fu07} J.X. Fu and S.T. Yau, A Monge-Amp`ere type equation motivated by string theory, Comm. Anal. Geom.
15 (2007), no. 1, 29-76.

\bibitem{Guan-Ren15} P. Guan, C. Ren and Z. Wang, Global $C^2$ estimates for convex
solutions of curvature equations, Comm. Pure Appl. Math. 68 (2015),
no. 8, 1287-1325.

\bibitem{Guan02} Guan, B.; Guan, P. Convex hypersurfaces of prescribed curvatures. Ann. of Math. (2) 156
(2002), no. 2, 655¨C673. doi:10.2307/3597202

\bibitem{Guan12} P. Guan, J. Li and Y.Y. Li, Hypersurfaces of Prescribed Curvature Measure,
Duke Math. J., 161, (2012),1927-1942.

\bibitem{Guan09} P. Guan, C.S. Lin and X. Ma, The Existence of Convex Body with
Prescribed Curvature Measures, International Mathematics Research
Notices, (2009), 1947-1975.

\bibitem{GuanB14} B. Guan, Second order estimates and regularity for fully nonlinear elliptic equations on Riemannian manifolds,
Duke Math. J. 163 (2014), 1491-1524.

\bibitem{Guan-Sp}  B. Guan, J. Spruck and L. Xiao, Interior curvature estimates and the asymptotic plateau problem in hyperbolic
space, J. Differential Geom. 96 (2014), no. 2, 201-222.

\bibitem{GT}D. Gilbarg and N.S. Trudinger:
Elliptic partial differential equations of second order. Reprint of the 1998 edition. Classics in Mathematics. Springer-Verlag, Berlin, 2001. xiv+517 pp.

\bibitem{GZ19} P.F. Guan and  X.W. Zhang:
A class of curvature type equations.
Preprint.arXiv:1909.03645.

\bibitem{HL} R. Harvey, H.B. Lawson, Calibrated geometries. Acta. Math.,
148 (1982), 47-157.

\bibitem{Iv1} N. Ivochkina, Solution of the Dirichlet problem for curvature
equations of order $m$, Mathematics of the USSR- Sbornik, 67,
(1990), 317-339. [32]

\bibitem{Iv2} N. Ivochkina. The Dirichlet problem for the equations of curvature
of order $m$, Leningrad Math. J. 2-3, (1991), 192-217.


\bibitem{Jin} Q. Jin and Y. Y. Li, Starshaped compact hypersurfaces with prescribed $k$-th mean curvature in hyperbolic space,
Discrete Contin. Dyn. Syst. 15 (2006), 2, 367-377.

\bibitem{Li-Ren} M. Li, C. Ren and Z. Wang, An interior estimate for convex
solutions and a rigidity theorem, J. Funct. Anal. 270 (2016),
2691-2714.

\bibitem{Li-Ol} Y. Y. Li, and V. I. Oliker, Starshaped compact hypersurfaces with prescribed
$m$-th mean curvature in elliptic space,
J. Partial Differential Equations 15 (2002), 3, 68-80.

\bibitem{Kry83}N. V. Krylov:
Boundedly inhomogeneous elliptic and parabolic equations in a domain,
Izv. Akad. Nauk SSSR Ser. Mat. 47,no 1 (1983), 75--108.

\bibitem{Kry95} N. V. Krylov:
On the general notion of fully nonlinear second order elliptic equation.Trans. Amer. Math. Soc. 347 (3),(1995), 857--895¡£

\bibitem{Li89}Y. Y. Li:
Degree theory for second order nonlinear elliptic operators and its applications.
Comm. Partial Differential Equations. 14(1989), 1541--1578.

\bibitem{Li-Sh} Q. R. Li and W. M. Sheng, Closed hypersurfaces with prescribed Weingarten
curvature in Riemannian manifolds, Calc. Var. Partial Differential
Equations 48 (2013), no. 1-2, 41-66.

\bibitem{LT94} M. Lin and N.S. Trudinger:
On some inequalities for elementary symmetric functions.
Bull. Aust. Math. Soc. 50(1994), 317--326.

\bibitem{Ol} V. I. Oliker, Hypersurfaces in $\mathbb{R}^{n+1}$ with prescribed Gaussian
curvature and related equations of Monge-Ampere type, Comm. Partial
Differential Equations, 9 (1984), 8, 807-838.

\bibitem{Ph19} D.H. Phong, S. Picard, and X.W. Zhang, On estimates for the Fu-Yau generalization of a Strominger
system, J. Reine Angew. Math. (Crelle¡¯s Journal), Vol. 2019, Issue 751 (2019), 243-274

\bibitem{Ph17} D.H. Phong, S. Picard, X.W. Zhang, The Fu-Yau equation with negative slope parameter, Invent.
Math., Vol. 209, No. 2 (2017), 541-576.

\bibitem{Ph20} D.H. Phong, S. Picard, and X.W. Zhang, Fu-Yau Hessian Equations, arXiv:1801.09842, to appear in
J. Differential Geometry.

\bibitem{Ren} C. Ren and Z. Wang, On the curvature estimates for Hessian equations, 2016, arXiv:1602.06535.

\bibitem{Sp} J. Spruck and L. Xiao, A note on starshaped compact hypersurfaces with prescribed scalar curvature in space form,
arXiv: 1505.01578.

\bibitem{Sch13}R. Schneider, Convex bodies: the Brunn-Minkowski theory, Second
edition, No. 151. Cambridge University Press, 2013.

\bibitem{Tr90} N.S. Trudinger:
The Dirichlet problem for the prescribed curvature equations.
Arch. Rational Mech. Anal. 111(2) (1990), 153--179.

\bibitem{Tr} A. Treibergs and W. Wei. Embedded hypersurfaces with prescribed mean curvature, J. Differential Geometry 18, 3
(1983), 513-521

\end{thebibliography}
\end{document}